\documentclass{amsart}
\usepackage{adjustbox,lipsum}
\usepackage{amsmath}
\usepackage{amssymb}
\usepackage{MnSymbol}
\usepackage[shortlabels]{enumitem}
\usepackage{tabu}
\usepackage{tikz-cd}
\usepackage{extarrows}
\usepackage{makecell}

\usepackage[
        colorlinks, citecolor=darkgreen,
        backref,
        pdfauthor={Samir Siksek},
]{hyperref}

\usepackage{comment}

\newcommand{\PP}{\mathbb{P}}
\newcommand{\Q}{\mathbb{Q}}

\newcommand{\Z}{\mathbb{Z}}

\newcommand{\fq}{\mathfrak{q}}

\newcommand{\OO}{\mathcal{O}}

\newcommand{\fp}{\mathfrak{p}}

\DeclareMathOperator{\Gal}{Gal}

\newtheorem{theorem}{Theorem}
\newtheorem{lemma}[theorem]{Lemma}

\newtheorem{proposition}[theorem]{Proposition}

\theoremstyle{definition}
\newtheorem{example}[theorem]{Example}
\newtheorem{definition}[equation]{Definition}

\theoremstyle{remark}
\newtheorem{remark}[theorem]{Remark}

\definecolor{darkgreen}{rgb}{0,0.5,0}

\begin{document}

\title[]{
del Pezzo surfaces with one bad prime over 
cyclotomic $\mathbb{Z}_\ell$-extensions
}

\begin{abstract}
Let $K$ be a number field and $S$ a finite set of primes of $K$. 
Scholl proved that there are only finitely many $K$-isomorphism classes of del
Pezzo surfaces of any degree $1 \le d \le 9$ over $K$ with good reduction away from $S$. 

Let instead $K$ be the cyclotomic $\Z_5$-extension of $\Q$.
In
this paper, we show,
for $d=3$, $4$, that there are infinitely many $\overline{\Q}$ isomorphism
classes of del Pezzo surfaces, defined over $K$, 
with good reduction away from the unique prime above $5$.  
\end{abstract}

\author{Maryam Nowroozi}
\address{Mathematics Institute\\
   University of Warwick\\
    CV4 7AL \\
    United Kingdom}

\email{maryam.nowroozi@warwick.ac.uk}

\date{\today}
\thanks{
Nowroozi is supported by the EPSRC studentship.}
\keywords{Shafarevich conjecture, Abelian varieties, del Pezzo surfaces, cyclic fields, cyclotomic fields.}
\makeatletter
\@namedef{subjclassname@2020}{%
\textup{2020} Mathematics Subject Classification}
\makeatother
\subjclass[2020]{11G35 (14G05).}

\maketitle
\section{Introduction}
Let $\ell$ be a rational
prime and $r$ a positive integer.
Write $\Q_{r,\ell}$ for the unique degree $\ell^r$ totally real subfield
of $\cup_{n=1}^\infty \Q(\mu_n)$, where $\mu_n$ denotes
the set of $\ell^n$-th roots of $1$. We let
$\Q_{\infty,\ell}=\cup_r \Q_{r,\ell}$;
this is called the $\Z_\ell$-cyclotomic extension of $\Q$,
and $\Q_{r,\ell}$ is called the $r$-th layer of $\Q_{\infty,\ell}$.
Now let $K$ be a number field, and write
$K_{\infty,\ell}=K \cdot \Q_{\infty,\ell}$
and $K_{r,\ell}=K \cdot \Q_{r,\ell}$.
To ease notation we shall sometimes write $K_\infty$
for $K_{\infty,\ell}$. 
Many theorems and conjectures regarding the arithmetic
of curves and varieties over number fields have
analogues over $K_\infty$. We mention some examples.
\begin{enumerate}[(I)]
\item The Mordell--Weil theorem asserts that for
an abelian variety $A$ over a number field $K$,
the Mordell--Weil group $A(K)$ is finitely generated.
A celebrated conjecture of Mazur \cite{Mazur_1972} asserts that $A(K_\infty)$
is finitely generated for an abelian variety $A$ over $K_\infty$.
\item Let $X$ be a curve of genus $\ge 2$ over a number field $K$. 
Faltings' theorem (previously known as the Mordell conjecture) \cite{Faltings} asserts that $X(K)$ is finite. A conjecture
of Parshin and Zarhin \cite[page 91]{Zarhin_Parshin}
asserts that $X(K_\infty)$ is finite, for a curve $X$ of genus $\ge 2$
over $K_\infty$.
\item
A theorem of Zarhin \cite[Corollary 4.2]{Zarhin_2010},
asserts that the Tate homomorphism conjecture
(also a theorem of Faltings \cite{Faltings} over number fields)
continues to hold over $K_\infty$.
\end{enumerate}

The Shafarevich conjecture for abelian varieties asserts that for a number field $K$, a finite set of primes $S$ of $K$, and a given dimension $d \ge 1$,
there are only finitely many principally polarized
abelian varieties $A$ defined over $K$
of dimension $d$ and good reduction away from $S$.
The Shafarevich conjecture was proved by Faltings \cite{Faltings}
in the same paper in which he proved the Tate homomorphism conjecture
and the Mordell conjecture. In a recent paper \cite{Visser},
Siksek and Visser show that the Shafarevich conjecture for abelian varieties does not hold over $K_\infty$. For example, they construct infinitely many principally polarized abelian surfaces 
defined over $\Q_{\infty,11}$ with good reduction away from
$2$, $11$, that are pairwise non-isomorphic over $\overline{\Q}$.

In view of the above theorems, conjectures and counterexamples,
it is interesting to test whether Shafarevich-type
conjectures for other families of varieties continue to hold
over $K_\infty$.
In 1985 Scholl proved the anaologue of the Shafarevich conjecture
for del Pezzo surfaces. More precisely,
Scholl \cite{Scholl} proved that, for any $1 \le d \le 9$,
and for any number field $K$ and for $S$ a finite set of primes of $K$,
then there are only finitely many $K$-isomorphism classes of del Pezzo surfaces
of degree $d$ over $K$ with good reduction away from $S$.
In this paper, we construct infinite families of del Pezzo surfaces over
$\mathbb{Q}_{\infty,5}$ with good reduction away from $5$.
\begin{theorem}\label{thm:degree3}
Let $r\geq 1$. Then there is a del Pezzo surface $X_{r}$ of degree $3$ 
defined over $\Q_{r,5}$ with good reduction away from the unique
prime above $5$. Moreover, within the family $\{X_r \; : \; r \ge 1\}$,
there is an infinite subfamily whose members are pairwise non-isomorphic
over $\overline{\Q}$.
\end{theorem}
\begin{theorem}\label{thm:degree4}
Let $r\geq 1$. Then there is a del Pezzo surface $Y_{r}$ of degree $4$ 
defined over $\Q_{r,5}$ with good reduction away from the unique
prime above $5$. Moreover, within the family $\{Y_r \; : \; r \ge 1\}$,
there is an infinite subfamily whose members are pairwise non-isomorphic
over $\overline{\Q}$.
\end{theorem}


We quickly sketch the construction of the surfaces $X_r$ and $Y_r$.
Let $\zeta=\zeta_{5^{r+1}}=\exp(2\pi i/5^{r+1})$. 
The extension $\Q(\zeta)/\Q$ is cyclic with Galois group
\[
\Gal(\Q(\zeta)/\Q) \; \cong \; \left(\Z/5^{r+1}\Z\right)^\times 
\; \cong \; \Z/4\Z \times \Z/5^r \Z.
\]
Let $\tau : \Q(\zeta) \rightarrow \Q(\zeta)$ be the field automorphism
given by
\[
\tau(\zeta)=\zeta^\delta
\]
where $\delta \in (\Z/5^{r+1}\Z)^\times$ is the element satisfying
\[
\delta^2 \equiv -1 \pmod{5^{r+1}}, \qquad \delta \equiv 2 \pmod{5};
\]
such $\delta$ exists and is unique by Hensel's lemma. 
We note that $\tau$ is an automorphism of order $4$.
Let
$H=\langle \tau \rangle$ be the subgroup of $\Gal(\Q(\zeta)/\Q)$
generated by $\tau$, and let $\Q_{r,5}$
by the fixed field of $H$ which has degree $5^r$ and 
Galois group isomorphic to $\Z/5^r \Z$. We note that
\[
\Q_{\infty,5}=\bigcup_{r=1}^\infty \Q_{r,5}
\]
Since $5$ is totally ramified in $\Q(\zeta)$, it is totally ramified
in $\Q_{r,5}$.

Let $P=(\zeta^2:\zeta:1)\in \PP^2(\Q(\zeta))$.
We let $X_{r}$ be the blow-up of $\PP^2$ in the six points 
\begin{equation}\label{eqn:six}
\lbrace \sigma(P): \sigma\in H \rbrace \cup \lbrace (1:0:0),(0:1:0)\rbrace.
\end{equation}
Let $Y_{r}$ be the blow-up of $\PP^2$ in the five points 
\begin{equation}\label{eqn:five}
\lbrace \sigma(P): \sigma\in H \rbrace \cup \lbrace (1:0:0)\rbrace.
\end{equation}
Since these sets of points are both stable 
under the action of $H$, the surfaces 
$X_{r}$ and $Y_{r}$ are defined over $\Q_{r,5}$.
\begin{lemma}\label{lem:generalposition}
The six points in \eqref{eqn:six} are in general position. Moreover,
for any prime ideal $\fq \subset \OO_{\Q_{r,5}}$ with $\fq \nmid 5$,
the reduction of the six points in \eqref{eqn:six} modulo
$\fq$ are in general position.
\end{lemma}
Lemma~\ref{lem:generalposition} is proved in Section~\ref{sec:construction}.
Since the six points in \eqref{eqn:six} are in general position,
the surface $X_r$ is a del Pezzo surface of degree $9-6=3$.
Moreover, 
for a prime $\fq \nmid 5$, as
the six points remain in general position modulo $\fq$,
it follows \cite[Section 3]{Scholl} 
that $X_r$ has good reduction modulo $\fq$.
To complete the proof of Theorem~\ref{thm:degree3}
we need to show that among the $X_r$ there are infinitely
many that are pairwise non-isomorphic over $\overline{\Q}$.
We do this in Section~\ref{sec:nonisomorphic} by using
the Clebsch--Salmon 
invariants of the degree $3$ del Pezzo surfaces $X_r$.

The five points in \eqref{eqn:five} are a subset
of the six points in \eqref{eqn:six} and it follows
that they are in general position, and their reduction
modulo any $\fq \nmid 5$ are in general position.
We conclude immediately that the $Y_r$
are degree $9-5=4$ del Pezzo surfaces with good reduction
away from the unique prime above $5$. In Section~\ref{sec:nonisomorphic}
we again use invariants to show that among the $Y_r$ there
are infinitely many that are pairwise non-isomorphic over $\overline{\Q}$. 

All the computations are done in Magma \cite{Magma} and SageMath \cite{sagemath}. You can find the code here: \url{https://github.com/maryam-nowroozi/del-pezzo-surfaces}

\section{background}
In this section, we review some of the definitions we will need in further sections.
Let $k$ be a field.
\begin{definition}
Let $1 \leq r \leq 8$ be an integer and let $P_1,\ldots, P_r$ be $r$ distinct points in $\mathbb{P}^2_k$. We say that $P_1,\ldots, P_r$ are in \emph{general position} if the following hold:
\begin{enumerate}[(i)]
    \item no three of the points lie on a line;
    \item no six of the points lie on a conic;
    \item no eight of the points lie on a singular cubic, with one of these eight points at the singular point.
\end{enumerate}
\end{definition}

\begin{definition}
Let $k$ be a field. A surface $X$ over $k$ is a del Pezzo surface if its anticanonical divisor $-K_{X}$ is ample.
\end{definition}
Note that if $-K_X$ is very ample, then it determines an embedding of $X$ into
$\mathbb{P}^n$ for some $n$, under which $-K_X$ corresponds to a hyperplane
section $H$. The degree of this embedding is $H^2$, which explains why the
integer $(-K_X)^2$ is called the degree of a del Pezzo surface $X$.

\begin{example}
Let $X=Q_1\cap Q_2$ be the intersection of two non-singular quadric hypersurfaces in $\mathbb{P}^4_k$. Then $X$ is a Del Pezzo surface of degree 4 (see \cite{pieropan}, Example 4.5). 
\end{example}

\begin{example}
Let $X$ be a cubic surface over $k$, i.e. a non-singular projective variety of dimension $2$ and degree $3$, 
then $X$ is a Del Pezzo surface of degree $3$ (see \cite{pieropan}, Example 4.6).
\end{example}

\begin{proposition}\label{dps}
Let $k$ be a field. If $X$ is a blow-up of $\mathbb{P}^2_k$ in $r\leq 8$ $k$-rational points in general position, then $X$ is a del Pezzo surface of 
degree $9-r$ over $k$.    
\end{proposition}

\begin{proof}
See \cite{pieropan}, Proposition 4.34.
\end{proof}

\begin{remark}
If $k$ is an algebraically closed field, then a del Pezzo surface $X$ is either isomorphic to $\mathbb{P}^1\times\mathbb{P}^1$  or 
$X$ arises as a blowing up of $\mathbb{P}^2_k$ in $r\leq8$ points in general position and $K_X^2=9-r$ (see \cite{pieropan}, Theorem 4.22). 
In general, this does not hold over an arbitrary field.  
\end{remark}

\section{Proof of Lemma~\ref{lem:generalposition}}\label{sec:construction} 

The goal of this section is to prove Lemma~\ref{lem:generalposition}. 
Let $k$ be a field and let $u$, $v \in k^\times$. Let
\[
T_{u,v} \; = \;
\{(u^2:u:1),~(u^{-2} : u^{-1},1),~(v^2:v:1),~(v^{-2}:v^{-1}:1),~(1:0:0),~(0:1:0)\}.
\]
We note that the set of six points in \eqref{eqn:six} is simply $T_{\zeta,\zeta^\delta}$.
\begin{lemma}\label{lem:Tuv}
Let
\[
f(u,v)=
		(u^4-1)(v^4-1)(u^2-v^2)(u^2 v^2-1).
\]
The set of six points $T_{u,v}$ is in general position if and only if $f(u,v) \ne 0$.
\end{lemma}
\begin{proof}
Let $M_{u,v}$ be the $6\times 6$ determinant with rows $(x^2,y^2,z^2,xy,xz,yz)$ with
$(x,y,z)$ running through the elements of $T_{u,v}$.  We see that
the six points of $T_{u,v}$ lie on a conic if and if this determinant is $0$.
A quick \texttt{Magma} computation shows that
\[
	M_{u,v}=(u+1)(u-1)(v+1)(v-1)(u-v)^2(uv-1)^2/u^3v^3.
\]

Moreover, three points $(x_i: y_i : z_i)$, with $i=1,2,3$, lie on a line if and only if the
$3 \times 3$ determinant with rows $(x_i,y_i,z_i)$ is zero. We computed this determinant
for the $\binom{6}{3}=20$ possible choices of triples of points from among six points
of $T_{u,v}$. All these determinants have numerators with factors  
among $u \pm 1$, $u^2+1$, $v \pm 1$, $v^2+1$,  $u \pm v$, $uv \pm 1$.
The lemma follows.
\end{proof}

\subsection*{Proof of Lemma~\ref{lem:generalposition}}
Let $r \ge 1$ and $\zeta=\exp(2\pi i/5^{r+1})$.
Write $u=\zeta$, $v=\zeta^\delta$. Then the six points in \eqref{eqn:six}
form the set $T_{u,v}$.
We consider the elements
\[
u^4-1, v^4-1, u^2-v^2, u^2 v^2-1.
\]
By Lemma~\ref{lem:Tuv}, the six points are in general position provided
these quantities are non-zero.
All of these, up to multiplication by a root of unity, have the form $\zeta^\mu - 1$,
where $\mu$ is among $4$, $4 \delta$, $2\delta-2$, $2 \delta+2$.
Since $\delta \equiv 2 \pmod{5}$, we see that $\mu \equiv 4$, $3$, $2$, $1 \pmod{5}$ respectively.
Thus in all cases $\zeta^\mu$ is a primitive $5^{r+1}$-th root of unity, and it follows that
the six points in \eqref{eqn:six} are in general position.

Next we consider the six points modulo a prime ideal $\fp$ of $\OO_{\Q_{r,5}}$. 
We note that, for the above four values of $\mu$, that the element $\zeta^\mu-1$ generates the unique
prime ideal of $\OO_{\Q(\zeta)}$ above $5$. Let $\fp$ be a prime ideal of $\OO_{\Q_{r,5}}$ 
satisfying $\fp \nmid 5$. 
Let $k=\OO_{\Q_{r,5}}/\fp$. Then the reduction of $\zeta^\mu-1$ in $\overline{k}$ is non-zero. 
Thus the reduction of six points in \eqref{eqn:six} in $\PP^2(\overline{k})$ are in general position. 
This completes the proof.

\section{Isomorphism of del Pezzo surfaces}\label{sec:nonisomorphic}
\subsection{Degree 3}
The goal of this section is to prove that there exists an infinite subfamily of $\lbrace X_r: r \geq 1\rbrace$ whose members are pairwise non-isomorphic. We will confirm this by using the Clebsch-Salmon invariants for cubic surfaces. 

We now recall the definition of Clebsch-Salmon invariants. Let $X$ be a cubic surface given by a system of equations of the form 
\begin{align*}
a_0x_0^3+a_1x_1^3+a_2x_2^3+a_3x_3^3+a_4x_4^3=0\\
x_0+x_1+x_2+x_3+x_4=0
\end{align*}
which is also called its \emph{pentahedral form}. The coefficients $a_0,a_1,a_2,a_3,a_4$ are called \emph{pentahedral coefficients} of the surface. 

\begin{definition}[Clebsch-Salmon invariants]\label{CS}
Let $\sigma_1,\ldots, \sigma_5$ be the elementary symmetric functions. We have the following invariants of degrees 8,16,24,32, and 40, respectively
\begin{align*}
&I_8=\sigma_4^2-4\sigma_3\sigma_5,\\
&I_{16}=\sigma_1\sigma_5^3,\\
&I_{24}=\sigma_4\sigma_5^4,\\
&I_{32}=\sigma_2\sigma_5^6,\\
&I_{40}=\sigma_5^8.
\end{align*}
\end{definition}

Let $X_{u,v}$ be the blow-up of $\mathbb{P}^2$ in $T_{u,v}$ and let $[I_8(X_{u,v}):I_{16}(X_{u,v}):I_{24}(X_{u,v}):I_{32}(X_{u,v}):I_{40}(X_{u,v})]\in\mathbb{P}^{8,16,24,32,40}$  be the invariants of $X_{u,v}$. The following lemma proves that there are at most finitely many del Pezzo surfaces isomorphic over $\overline{\mathbb{Q}}$ to $X_{u,v}$.

\begin{lemma}
Let $(u,v)\in \overline{\mathbb{Q}}$, $f(u,v)\neq 0$. Then, there are at most finitely many pairs $(r,s) \in \overline{\mathbb{Q}}$, $f(r,s)\neq 0$ such that $X_{r,s}$ is isomorphic over $\overline{\mathbb{Q}}$ to $X_{u,v}$
\end{lemma}

\begin{proof}
Suppose $X_{r,s}\cong_{\overline{\mathbb{Q}}}X_{u,v}$. Then 
\begin{equation*}
\frac{I_{16}(X_{u,v})}{I_8^2(X_{u,v})}=\frac{I_{16}(X_{r,s})}{I_8^2(X_{r,s})}\ \ \text{and}\ \  \frac{I_{24}(X_{u,v})}{I_8^3(X_{u,v})}=\frac{I_{24}(X_{r,s})}{I_8^3(X_{r,s})}
\end{equation*}

Let 
\begin{align*}
&F(r,s)=I_{16}(X_{u,v})-\alpha I_8^2(X_{u,v})\in \mathbb{Q}(\alpha)[r,s]\\
&G(r,s)=I_{24}(X_{u,v})-\beta I_8^3(X_{u,v})\in \mathbb{Q}(\beta)[r,s]
\end{align*}

Then, $F(r,s)=G(r,s)=0$. Let $h_1(r)=\operatorname{Res}_{s}(F(r,s),G(r,s))$ and $h_2(s)=\operatorname{Res}_{r}(F(r,s),G(r,s))$. Then $h_1(r)=h_2(s)=0$. So, the number of possibilities for the pair $(r,s)$ is at most $\deg(h_1)\cdot\deg(h_2)$. A quick \texttt{Magma} computation shows that $\deg(h_1)\cdot\deg(h_2)=518400$.
\end{proof}

This proves that within the family $\{X_r \; : \; r \ge 1\}$,
there is an infinite subfamily whose members are pairwise non-isomorphic
over $\overline{\Q}$.

\subsection{degree 4}
Recall that a Del Pezzo surface of degree 4 is a complete intersection of two quadrics in $\mathbb{P}^4$. Let $X\subset \mathbb{P}^4$ be a smooth surface defined by the intersection of two quadrics over an algebraically closed field $k$ of characteristic different from 2.  We know that $X$ is charactrized upto isomorphism by the degeneracy locus of the pencil of quadrics containing $X$, i.e., by the form  $$h(t,w)=\det(tA_F+wA_G),$$ where $A_F$ and $A_G$ are symmetric $5\times5$ matrices whose associated quadric forms define $X$ (see \cite{modulidegree4del} for more details).

The two-dimensional space of binary quintic forms with non-vanishing discriminant up to linear change of variables serves as a moduli space of smooth del Pezzo surface of degree 4.

The Del Pezzo surface defined by blowing up the five points 
\begin{align*}
&P_1=(u^2:u:1),\\
&P_2=(u^{-2}:u^{-1}:1),\\
&P_3=(v^2:v:1),\\
&P_4=(v^{-2}:v^{-1}:1),\\
&P_5=(1:0:0).
\end{align*}
is a complete intersection of two quadrics in $\mathbb{P}^4$ given below
\begin{align*}
    &F(x_0,x_1,x_2,x_3,x_4)=-x_1x_2+x_2^2+x_0x_3+ax_2x_3+bx_3^2+ax_3x_4+x_4^2,
    \\&G(x_0,x_1,x_2,x_3,x_4)=-x_1x_3-ax_3^2+x_0x_4+ax_2x_4.& 
\end{align*}
where $a=\frac{-u^2v-uv^2-u-v}{uv}$
and $b=\frac{u^2v^2+u^2+2uv+v^2+1}{uv}$.

The defined matrices $A_F$ and $A_G$ are as follows: 

\[
A_F=
\begin{bmatrix}
0 & 0 & 0 & 1/2 & 0\\
0 & 0 & -1/2 & 0 & 0\\
0 & -1/2 & 1 & a/2 & 0\\
1/2 & 0 & a/2 & b & a/2\\
0 & 0 & 0 & a/2 & 1
\end{bmatrix}
,\  
A_G=
\begin{bmatrix}
0 & 0 & 0 & 0 & 1/2\\
0 & 0 & 0 & -1/2 & 0\\
0 & 0 & 0 & 0 & a/2\\
0 & -1/2 & 0 & -a & 0\\
1/2 & 0 & a/2 & 0 & 0
\end{bmatrix}
\]
We have 
\[h(t,w)=\det(tA_F+wA_G)=
\begin{vmatrix}
0 & 0 & 0 & \frac{1}{2}t & \frac{1}{2}w\\
0 & 0 & -\frac{1}{2}t & -\frac{1}{2}w & 0\\
0 & -\frac{1}{2}t & t & \frac{a}{2}t & \frac{a}{2}w\\
\frac{1}{2}t & -\frac{1}{2}w & \frac{a}{2}t & -aw+bt & \frac{a}{2}t\\
\frac{1}{2}w & 0 & \frac{a}{2}w & \frac{a}{2}t & t
\end{vmatrix}
\]

So, we get $h(t,w)=\frac{1}{16}(t^5-at^4w+bt^3w^2-at^2w^3+tw^4)$ which is a binary quintic. The ring $SL_2$-invariants of binary quintics is generated by four invariants, $I_4, I_8, I_{12}$, and, $I_{18}$ (see \cite{schur}).

Let $Y_{u,v}$ be the blow-up of $\mathbb{P}^2$ in five points $P_1,\ldots, P_5$ and let $[I_4:I_8:I_{12}:I_{18}]\in\mathbb{P}^{4,8,12,18}$  be the invariants of the degeneracy locus of the pencil of quadrics containing $Y_{u,v}$, defined above. The following lemma proves that there are at most finitely many del Pezzo surfaces isomorphic over $\overline{\mathbb{Q}}$ to $Y_{u,v}$.

\begin{lemma}
Let $(u,v)\in \overline{\mathbb{Q}}$, $f(u,v)\neq 0$. Then, there are at most finitely many pairs $(r,s) \in \overline{\mathbb{Q}}$, $f(r,s)\neq 0$ such that $Y_{r,s}$ is isomorphic over $\overline{\mathbb{Q}}$ to $Y_{u,v}$
\end{lemma}

\begin{proof}
Suppose $Y_{r,s}\cong_{\overline{\mathbb{Q}}}Y_{u,v}$. Then 
\begin{equation*}
\frac{I_8(Y_{u,v})}{I_4^2(Y_{u,v})}=\frac{I_8(Y_{r,s})}{I_4^2(Y_{r,s})}\ \ \text{and}\ \  \frac{I_{12}(Y_{u,v})}{I_4^3(Y_{u,v})}=\frac{I_{12}(Y_{r,s})}{I_4^3(Y_{r,s})}
\end{equation*}

Let 
\begin{align*}
&F(r,s)=I_8(Y_{u,v})-\alpha I_4^2(Y_{u,v})\in \mathbb{Q}(\alpha)[r,s]\\
&G(r,s)=I_{12}(Y_{u,v})-\beta I_4^3(Y_{u,v})\in \mathbb{Q}(\beta)[r,s]
\end{align*}

Then, $F(r,s)=G(r,s)=0$. Let $h_1(r)=\operatorname{Res}_{s}(F(r,s),G(r,s))$ and $h_2(s)=\operatorname{Res}_{r}(F(r,s),G(r,s))$. Then $h_1(r)=h_2(s)=0$. So, the number of possibilities for the pair $(r,s)$ is at most $\deg(h_1)\cdot\deg(h_2)$. A quick \texttt{Magma} computation shows that $\deg(h_1)\cdot\deg(h_2)=518400$.
\end{proof}

This proves that within the family $\{Y_r \; : \; r \ge 1\}$,
there is an infinite subfamily whose members are pairwise non-isomorphic
over $\overline{\Q}$.
\bibliography{ref}
\bibliographystyle{plain}
\end{document}